\documentclass[12pt]{article}
\usepackage{epsf}

\usepackage{amsthm}
\usepackage{amssymb}
\usepackage{amsmath}

\usepackage[top=1.5cm,bottom=1.5cm,left=2.5cm,right=2.5cm,includehead,includefoot
           ]{geometry}

\makeatletter

\numberwithin{equation}{section}
\newtheorem{theorem}{Theorem}[section]
\newtheorem{lemma}[theorem]{Lemma}
\newtheorem{corollary}[theorem]{Corollary}

\title{Characterizations of centralizers and derivations on some algebras}
\author{\begin{tabular}{c} Jun He, Jiankui Li\footnote{Corresponding author.
E-mail address: jiankuili@yahoo.com},  and Wenhua Qian
\\{\small\it Department of Mathematics, East China University of
Science and Technology}\\
{\small\it Shanghai 200237, China}
\end{tabular}}
\date{}
\begin{document}
\maketitle \abstract
A linear mapping $\phi$ on an algebra $\mathcal{A}$ is called
a centralizable mapping at $G\in\mathcal{A}$ if $\phi(AB)=\phi(A)B=A\phi(B)$
for each $A$ and $B$ in $\mathcal{A}$ with $AB=G$, and $\phi$ is called a
derivable mapping at $G\in\mathcal{A}$ if $\phi(AB)=\phi(A)B+A\phi(B)$
for each $A$ and $B$ in $\mathcal{A}$ with $AB=G$. A point $G$ in $\mathcal{A}$ is called a
full-centralizable point (resp. full-derivable point) if every centralizable (resp. derivable)
mapping at $G$ is a centralizer (resp. derivation).
We prove that every point in a von Neumann algebra or a triangular algebra
is a full-centralizable point.
We also prove that a point in a von Neumann algebra is a full-derivable point
if and only if its central carrier is the unit.

\

\

\
\section{Introduction}\

Let $\mathcal{A}$ be an associative algebra over the complex field $\mathbb{C}$, and $\phi$ be a linear mapping from $\mathcal{A}$ into itself.
$\phi$ is called a \emph{centralizer} if $\phi(AB)=\phi(A)B=A\phi(B)$ for each $A$ and $B$ in $\mathcal{A}$.
Obviously, if $\mathcal{A}$ is an algebra with unit $I$, then $\phi$ is a centralizer if and only if $\phi(A)=\phi(I)A=A\phi(I)$ for every $A$ in $\mathcal{A}$.
$\phi$ is called a \emph{derivation} if $\phi(AB)=\phi(A)B+A\phi(B)$ for each $A$ and $B$ in $\mathcal{A}$.

A linear mapping $\phi:\mathcal{A}\rightarrow\mathcal{A}$ is called
a \emph{centralizable mapping at $G\in\mathcal{A}$} if $\phi(AB)=\phi(A)B=A\phi(B)$
for each $A$ and $B$ in $\mathcal{A}$ with $AB=G$, and $\phi$ is called a
\emph{derivable mapping at $G\in\mathcal{A}$} if $\phi(AB)=\phi(A)B+A\phi(B)$
for each $A$ and $B$ in $\mathcal{A}$ with $AB=G$. An element $G$ in $\mathcal{A}$ is called a
\emph{full-centralizable point (resp. full-derivable point)} if every centralizable (resp. derivable)
mapping at $G$ is a centralizer (resp. derivation).

In \cite{501}, Bre\v{s}ar proves that if $\mathcal{R}$ is a prime ring with a nontrival
idempotent, then $0$ is a full-centralizable point. In \cite{502}, X. Qi and J. Hou
characterize centralizable and derivable
mappings at 0 in triangular algebras. In \cite{503}, X. Qi proves that every nontrival
idempotent in a prime ring is a full-centralizable point. In \cite{519}, W. Xu, R. An and J. Hou
prove that every element in $B(\mathcal{H})$ is a full-centralizable point, where $\mathcal{H}$ is a Hilbert space.
For more information on centralizable and derivable
mappings, we refer to \cite{509,505,506,508,504,507}.

For a von Neumann algebra $\mathcal{A}$, the \emph{central carrier} $\mathcal{C}(A)$ of an element $A$ in $\mathcal{A}$
is the projection $I-P$, where $P$ is the union of all central projections $P_{\alpha}$ in $\mathcal{A}$ such that $P_{\alpha}A=0$.

This paper is organized as follows. In Section 2,
by using the techniques about central carriers, we show that every element in a von Neumann algebra is a full-centralizable point.

Let $\mathcal{A}$ and $\mathcal{B}$ be two unital algebras over the complex field $\mathbb{C}$, and $\mathcal{M}$ be a unital $(\mathcal{A}, \mathcal{B})$-bimodule
which is faithful both as a left $\mathcal A$-module and a right $\mathcal B$-module. The algebra

$$Tri(\mathcal{A},\mathcal{M},\mathcal{B})=\left\{\left[
                   \begin{array}{cc}
                     A & M \\
                     0 & B \\
                   \end{array}
                 \right]:A\in\mathcal{A}, B\in\mathcal{B}, M\in\mathcal{M}\right\}$$
under the usual matrix addition and matrix multiplication is called a \emph{triangular algebra}.

In Section 3, we show that if $\mathcal{A}$ and $\mathcal{B}$ are two unital Banach algebras,
then every element in $Tri(\mathcal{A},\mathcal{M},\mathcal{B})$ is a full-centralizable point.

In Section 4, we show that for every point $G$ in a von Neumann algebra $\mathcal{A}$,
if $\Delta$ is a derivable mapping at $G$, then $\Delta=D+\phi$, where $D:\mathcal{A}\rightarrow\mathcal{A}$ is a derivation
and $\phi:\mathcal{A}\rightarrow\mathcal{A}$ is a centralizer.
Moreover, we prove that $G$ is a full-derivable point if and only if $\mathcal{C}(G)=I$.

\section{Centralizers on von Neumann algebras}\

In this section, $\mathcal{A}$ denotes a unital algebra and $\phi:\mathcal{A}\rightarrow\mathcal{A}$
is a centralizable mapping at a given point $G\in\mathcal A$. The main result is the following theorem.

\begin{theorem}
Let $\mathcal{A}$ be a von Neumann algebra acting on a Hilbert space $\mathcal H$. Then every element $G$ in $\mathcal{A}$ is a full-centralizable point.
\end{theorem}

Before proving Theorem 2.1, we need the following several lemmas.

\begin{lemma}
Let $\mathcal{A}$ be a unital Banach algebra with the form $\mathcal{A}=\sum\limits_{i\in\Lambda}\bigoplus\mathcal{A}_i$. Then $\phi(\mathcal{A}_i)\subseteq\mathcal{A}_i$. Moreover, suppose $G=\sum\limits_{i\in\Lambda}G_i$, where $G_i\in\mathcal{A}_i$. If $G_i$ is a full-centralizable point in $\mathcal{A}_i$ for every $i\in \Lambda$, then $G$ is a full-centralizable point in $\mathcal{A}$.
\end{lemma}

\begin{proof}
Let $I_i$ be the unit in $\mathcal{A}_i$. Suppose that $A_i$ is an invertible element in $\mathcal{A}_i$, and $t$ is an arbitrary nonzero element in $\mathbb{C}$. It is easy to check that
$$(I-I_i+t^{-1}GA_i^{-1})((I-I_i)G+tA_i)=G.$$
So we have
$$(I-I_i+t^{-1}GA_i^{-1})\phi((I-I_i)G+tA_i)=\phi(G).$$
Considering the coefficient of $t$, since $t$ is arbitrarily chosen, we have $(I-I_i)\phi(A_i)=0$. It follows that $\phi(A_i)=I_i\phi(A_i)\in\mathcal{A}_i$ for all invertible elements. Since $\mathcal{A}_i$ is a Banach algebra, every element can be written into the sum of two invertible elements. So the above equation holds for all elements in $\mathcal{A}_i$. That is to say $\phi(\mathcal{A}_i)\subseteq\mathcal{A}_i$.

Let $\phi_i=\phi \mid_{\mathcal{A}_i}$. For every $A$ in $\mathcal{A}$, we write $A=\sum\limits_{i\in\Lambda}A_i$. Assume $AB=G$. Since $A_iB_i=G_i$ and $\phi(\mathcal{A}_i)\subseteq\mathcal{A}_i$, we have
$$\sum\limits_{i\in\Lambda}\phi(G_i)=\sum\limits_{i\in\Lambda}\phi(A_i)\sum\limits_{i\in\Lambda} B_i=\sum\limits_{i\in\Lambda}\phi(A_i)B_i.$$
It implies that $\phi_i(G_i)=\phi_i(A_i)B_i$. Similarly, we can obtain $\phi_i(G_i)=A_i\phi_i(B_i)$. By assumption, $G_i$ is a full-centralizable point, so $\phi_i$ is a centralizer. Hence
$$\phi(A)=\sum\limits_{i\in\Lambda}\phi_i(A_i)=\sum\limits_{i\in\Lambda}\phi_i(I_i)A_i=\sum\limits_{i\in\Lambda}\phi_i(I_i)\sum\limits_{i\in\Lambda}A_i=\phi(I)A.$$
Similarly, we can prove $\phi(A)=A\phi(I)$.
Hence $G$ is a full-centralizable point.
\end{proof}

\begin{lemma}
Let $\mathcal{A}$ be a $C^*$-algebra. If $G^*$ is a full-centralizable point in $\mathcal{A}$, then $G$ is a full-centralizable point in $\mathcal{A}$.
\end{lemma}

\begin{proof}
Define a linear mapping $\widetilde{\phi}:\mathcal{A}\rightarrow\mathcal{A}$ by: $\widetilde{\phi}(A)=(\phi(A^*))^*$
for every $A$ in $\mathcal{A}$.
For each $A$ and $B$ in $\mathcal{A}$ with $AB=G^*$, we have $B^*A^*=G$. It follows that $\phi(G)=\phi(B^*)A^*=B^*\phi(A^*)$. By the definition of $\widetilde{\phi}$,
we obtain $\widetilde{\phi}(G^*)=\widetilde{\phi}(A)B=A\widetilde{\phi}(B)$.
Since $G^{*}$ is a full-centralizable point in $\mathcal{A}$, we have that $\widetilde{\phi}$ is a centralizer. Thus
$\phi$ is also a centralizer. Hence $G$ is a full-centralizable point in $\mathcal{A}$.
\end{proof}

For a unital algebra $\mathcal{A}$ and a unital $\mathcal{A}$-bimodule $\mathcal{M}$,
an element $A\in\mathcal{A}$ is called a \emph{left separating point (resp. right separating point)} of $\mathcal{M}$
if $AM=0$ implies $M=0$ ($MA=0$ implies $M=0$) for every $M\in\mathcal{M}$.

\begin{lemma}
Let $\mathcal{A}$ be a unital Banach algebra and $G$ be a left and
right separating point in $\mathcal{A}$. Then $G$ is a full-centralizable point.
\end{lemma}

\begin{proof}
For every invertible element $X$ in $\mathcal A$, we have
$$\phi(I)G=\phi(G)=\phi(XX^{-1}G)=\phi(X)X^{-1}G.$$
Since $G$ is a right separating point, we obtain $\phi(I)=\phi(X)X^{-1}$. It follows that $\phi(X)=\phi(I)X$ for each invertible element $X$ and so for all elements in $\mathcal A$. Similarly, we have that $\phi(X)=X\phi(I)$.
Hence $G$ is a full-centralizable point.
\end{proof}

\begin{lemma}
Let $\mathcal{A}$ be a von Neumann algebra. Then $G=0$ is a full-centralizable point.
\end{lemma}

\begin{proof}
For any projection $P$ in $\mathcal{A}$, since $P(I-P)=(I-P)P=0$, we have
$$\phi (P)(I-P)=P\phi(I-P)=\phi(I-P)P=(I-P)\phi(P)=0.$$ It follows that $\phi(P)=\phi(I)P=P\phi(I)$. By \cite[Proposition 2.4]{515} and \cite[Corollary 1.2]{516},
we know that $\phi$ is continuous. Since
$\mathcal{A}=\overline{span\{P\in\mathcal{A}:P=P^*=P^2\}}$, it follows that $\phi(A)=\phi(I)A=A\phi(I)$ for every $A\in\mathcal{A}$. Hence $G$ is a full-centralizable point.
\end{proof}

\begin{lemma}
Let $\mathcal{A}$ be a von Neumann algebra acting on a Hilbert space $\mathcal H$ and $P$ be the range projection of $G$ .
If $\mathcal{C}(P)=\mathcal{C}(I-P)=I$, then $G$ is a full-centralizable point.

\end{lemma}

\begin{proof}
Set $P_1=P,~ P_2=I-P$, and denote $P_i\mathcal{A}P_j$ by $\mathcal{A}_{ij}$, $i,j=1,2$. For every $A$ in $\mathcal{A}$, denote $P_iAP_j$ by $A_{ij}$.

Firstly, we claim that the condition $A\mathcal{A}_{ij}=0$ implies $AP_i=0$, and similarly, $\mathcal{A}_{ij}A=0$ implies $P_jA=0$.
Indeed, since $\mathcal{C}(P_j)=I$, by \cite[Proposition 5.5.2]{514}, the range of $\mathcal{A}P_j$ is dense in $\mathcal H$. So $AP_i\mathcal{A}P_j=0$ implies $AP_i=0$. On the other hand, if $\mathcal{A}_{ij}A=0$, then $A^{*}\mathcal{A}_{ji}=0$. Hence $A^{*}P_j=0$ and $P_jA=0$.

Besides, since $P_1=P$ is the range projection of $G$, we have $P_1G=G$. Moreover, if $AG=0,$ then $AP_1=0$.

In the following, we assume that $A_{ij}$ is an arbitrary element in $\mathcal{A}_{ij}$, $i,j=1,2$, and $t$ is an arbitrary nonzero element in $\mathbb{C}$. Without loss of generality, we may assume that $A_{11}$ is invertible in $\mathcal{A}_{11}$.

\textbf{Claim 1} $\phi(\mathcal{A}_{12})\subseteq\mathcal{A}_{12}$.

Since $(P_1+tA_{12})G=G$, we have $\phi(G)=\phi(P_1+tA_{12})G$. It implies that $\phi(A_{12})G=0$. Hence
$\phi(A_{12})P_1=0.$

By $(P_1+tA_{12})G=G$, we also have $\phi(G)=(P_1+tA_{12})\phi(G)$. It follows that $A_{12}\phi(G)=A_{12}\phi(P_1)G=0$. So $A_{12}\phi(P_1)P_1=0$. Hence $P_2\phi(P_1)P_1=0.$

Since
$(A_{11}+tA_{11}A_{12})(A_{11}^{-1}G-A_{12}A_{22}+t^{-1}A_{22})=G,$
we have
\begin{align}
\phi(A_{11}+tA_{11}A_{12})(A_{11}^{-1}G-A_{12}A_{22}+t^{-1}A_{22})=\phi(G).\label{203}
\end{align}
Since $t$ is arbitrarily chosen in \eqref{203}, we obtain $$\phi(A_{11})(A_{11}^{-1}G-A_{12}A_{22})+\phi(A_{11}A_{12})A_{22}=\phi(G).$$
Since $A_{12}$ is also arbitrarily chosen, we can obtain $\phi(A_{11})A_{12}A_{22}=\phi(A_{11}A_{12})A_{22}.$  Taking $A_{22}=P_2$, since $\phi(A_{12})P_1=0$, we have
\begin{align}
\phi(A_{11}A_{12})=\phi(A_{11})A_{12}.\label{205}
\end{align}
Taking $A_{11}=P_1$, since $P_2\phi(P_1)P_1=0,$ we have
\begin{align}
P_2\phi(A_{12})=P_2\phi(P_1)A_{12}=0.\label{207}
\end{align}
So
$$\phi(A_{12})=\phi(A_{12})P_1+P_1\phi(A_{12})P_2+P_2\phi(A_{12})P_2
=P_1\phi(A_{12})P_2\subseteq\mathcal{A}_{12}. $$

\textbf{Claim 2} $\phi(\mathcal{A}_{11})\subseteq\mathcal{A}_{11}$.

Considering the coefficient of $t^{-1}$ in \eqref{203}, we have $\phi(A_{11})A_{22}=0.$ Thus
$\phi(A_{11})P_2=0.$
By \eqref{205}, we obtain $P_2\phi(A_{11})A_{12}=P_2\phi(A_{11}A_{12})=0.$ It follows that
$P_2\phi(A_{11})P_1=0.$
Therefore, $\phi(A_{11})=P_1\phi(A_{11})P_1\subseteq\mathcal{A}_{11}. $

\textbf{Claim 3} $\phi(\mathcal{A}_{22})\subseteq\mathcal{A}_{22}$.

By
$(A_{11}+tA_{11}A_{12})(A_{11}^{-1}G-A_{12}A_{22}+t^{-1}A_{22})=G,$
we also have
$$(A_{11}+tA_{11}A_{12})\phi(A_{11}^{-1}G-A_{12}A_{22}+t^{-1}A_{22})=\phi(G).$$
Through a similar discussion to equation \eqref{203}, we can prove
$P_1\phi(A_{22})=0$
and
\begin{align}
\phi(A_{12}A_{22})=A_{12}\phi(A_{22}).\label{211}
\end{align}
Thus $A_{12}\phi(A_{22})P_1=\phi(A_{12}A_{22})P_1=0$.
It follows that
$P_2\phi(A_{22})P_1=0.$
Therefore, $\phi(A_{22})=P_2\phi(A_{22})P_2\subseteq\mathcal{A}_{22}. $

\textbf{Claim 4} $\phi(\mathcal{A}_{21})\subseteq\mathcal{A}_{21}$.

Since
$(A_{11}+tA_{11}A_{12})(A_{11}^{-1}G-A_{12}A_{21}+t^{-1}A_{21})=G,$
we have
$$(A_{11}+tA_{11}A_{12})\phi(A_{11}^{-1}G-A_{12}A_{21}+t^{-1}A_{21})=\phi(G).$$
According to this equation, we can similarly obtain that
$P_1\phi(A_{21})=0$
and
\begin{align}
A_{12}\phi(A_{21})=\phi(A_{12}A_{21}).\label{219}
\end{align}
Hence $A_{12}\phi(A_{21})P_2=\phi(A_{12}A_{21})P_2=0$.
It follows that $P_2\phi(A_{21})P_2=0$. Therefore, $\phi(\mathcal{A}_{21})=P_2\phi(A_{21})P_1\subseteq\mathcal{A}_{21}$.

\textbf{Claim 5} $\phi(A_{ij})=\phi(P_i)A_{ij}=A_{ij}\phi(P_j)$ for each $i,j\in \{1,2\}$.

By taking $A_{11}=P_1$ in \eqref{205}, we have $\phi(A_{12})=\phi(P_1)A_{12}$. By taking $A_{22}=P_2$ in \eqref{211}, we have $\phi(A_{12})=A_{12}\phi(P_2)$.

By \eqref{205}, we have $\phi(A_{11})A_{12}=\phi(A_{11}A_{12})=\phi(P_1)A_{11}A_{12}$. It follows that
$\phi(A_{11})=\phi(P_1)A_{11}.$ On the other hand, $\phi(A_{11})A_{12}=\phi(A_{11}A_{12})=A_{11}A_{12}\phi(P_2)=A_{11}\phi(A_{12})=A_{11}\phi(P_1)A_{12}$.
It follows that $\phi(A_{11})=A_{11}\phi(P_1).$

By \eqref{211} and \eqref{219}, through a similar discussion as above, we can obtain that
$\phi(A_{22})=A_{22}\phi(P_2)=\phi(P_2)A_{22}$
and
$\phi(A_{21})=A_{21}\phi(P_1)=\phi(P_2)A_{21}.$

Now we have proved that
$\phi(\mathcal{A}_{ij})\subseteq\mathcal{A}_{ij}$
and $\phi(A_{ij})=\phi(P_i)A_{ij}=A_{ij}\phi(P_j).$
It follows that
\begin{align*}
\phi(A)&=\phi(A_{11}+A_{12}+A_{21}+A_{22})\notag\\
&=\phi(P_1)(A_{11}+A_{12}+A_{21}+A_{22})+\phi(P_2)(A_{11}+A_{12}+A_{21}+A_{22})\notag\\
&=\phi(P_1+P_2)(A_{11}+A_{12}+A_{21}+A_{22})\notag\\
&=\phi(I)A.\notag
\end{align*}
Similarly, we can prove that $\phi(A)=A\phi(I)$. Hence $G$ is a full-centralizable point.
\end{proof}

\begin{proof}[\bf Proof of Theorem 2.1]

Suppose the range projection of $G$ is $P$.
Set $Q_1=I-\mathcal{C}(I-P)$, $Q_2=I-\mathcal{C}(P)$, and $Q_3=I-Q_1-Q_2$. Since $Q_1\leq P$ and $Q_2\leq I-P$, $\{Q_i\}_{i=1,2,3}$ are mutually orthogonal central projections.
Therefore $\mathcal{A}=\sum\limits_{i=1}^{3}\bigoplus\mathcal{A}_i=\sum\limits_{i=1}^{3}\bigoplus(Q_i\mathcal{A})$. Obviously, $\mathcal{A}_i$ is also a von Neumann algebra acting on $Q_i\mathcal H$. For each element $A$ in $\mathcal{A}$, we write
$A=\sum\limits_{i=1}^{3}A_i=\sum\limits_{i=1}^{3}Q_iA$.

We divide our proof into two cases.\\
\textbf{Case 1} $ker(G)=\{0\}$

Since $Q_1\leq P$, we have $\overline{ran G_1}=\overline{ran Q_1G}=Q_1\mathcal H$. Since $G$ is injective on $\mathcal H$, $G_1=Q_1G$ is also injective on $Q_1\mathcal H$. Hence $G_1$ is a separating point(both right and left) in $\mathcal{A}_1$. By Lemma 2.4, $G_1$ is a full-centralizable point in $\mathcal{A}_1$.

Since $Q_2\leq I-P$, we have $G_2=Q_2G=0$. By Lemma 2.5, $G_2$ is a full-centralizable point in $\mathcal{A}_2$.

Note that $\overline{ran G_3}=\overline{ran Q_3G}=Q_3P=P_3$. Denote the central carrier of $P_3$ in $\mathcal{A}_3$ by $\mathcal{C}_{\mathcal{A}_3}(P_3)$. We have
$Q_3-\mathcal{C}_{\mathcal{A}_3}(P_3)\leq Q_3-P_3=Q_3(I-P)\leq I-P$. Obviously, $Q_3-\mathcal{C}_{\mathcal{A}_3}(P_3)$ is a central projection orthogonal to $Q_2$,
so $Q_3-\mathcal{C}_{\mathcal{A}_3}(P_3)+I-\mathcal{C}(P)\leq I-P$. That is $Q_3-\mathcal{C}_{\mathcal{A}_3}(P_3)+P\leq \mathcal{C}(P)$. It implies that $Q_3-\mathcal{C}_{\mathcal{A}_3}(P_3)=0$, i.e. $\mathcal{C}_{\mathcal{A}_3}(P_3)=Q_3$. Similarly, we can prove $\mathcal{C}_{\mathcal{A}_3}(Q_3-P_3)=Q_3$. By Lemma 2.6, $G_3$ is a full-centralizable point in $\mathcal{A}_3$.

By Lemma 2.2, $G$ is a full-centralizable point.\\
\textbf{Case 2} $ker(G)\neq \{0\}$

In this case, $G_2$ and $G_3$ are still full-centralizable points. Since $\overline{ran G_1}=Q_1H$ , we have $ker(G_1^{*})=\{0\}$. By Case 1, $G_1^{*}$ is a full-centralizable point in $\mathcal{A}_1$. By Lemma 2.3, $G_1$ is also a full-centralizable point in $\mathcal{A}_1$.

By Lemma 2.2, $G$ is a full-centralizable point.
\end{proof}

\section{Centralizers on triangular algebras}\

In this section, we characterize the full-centralizable points on triangular algebras. The following theorem is our main result.

\begin{theorem}
Let $\mathcal{J}=\left[
                   \begin{array}{cc}
                     \mathcal{A} & \mathcal{M} \\
                     0 & \mathcal{B} \\
                   \end{array}
                 \right]
$
be a triangular algebra, where $\mathcal{A}$ and $\mathcal{B}$ are two unital Banach algebras.
Then every $G$ in $\mathcal J$ is a full-centralizable point.
\end{theorem}

\begin{proof}
Let $\phi: \mathcal{J}\rightarrow\mathcal{J}$ be a centralizable mapping at $G$.

Since $\phi$ is linear, for every
$\left[
  \begin{array}{cc}
    X & Y \\
    0 & Z \\
  \end{array}
\right]$
in $\mathcal{J}$, we write

$$\phi\left[
       \begin{array}{cc}
X & Y \\
0 & Z  \\
       \end{array}
     \right]
=\left[
   \begin{array}{cc}
f_{11}(X)+g_{11}(Y)+h_{11}(Z) & f_{12}(X)+g_{12}(Y)+h_{12}(Z) \\
0 & f_{22}(X)+g_{22}(Y)+h_{22}(Z)  \\
   \end{array}
 \right],
$$
where $f_{11}:\mathcal{A}\rightarrow\mathcal{A},~f_{12}:\mathcal{A}\rightarrow\mathcal{M},~f_{22}:\mathcal{A}\rightarrow\mathcal{B},~
g_{11}:\mathcal{M}\rightarrow\mathcal{A},~g_{12}:\mathcal{M}\rightarrow\mathcal{M},$
$g_{22}:\mathcal{M}\rightarrow\mathcal{B},~
h_{11}:\mathcal{B}\rightarrow\mathcal{A},~h_{12}:\mathcal{B}\rightarrow\mathcal{M},~h_{22}:\mathcal{B}\rightarrow\mathcal{B},$
are all linear mappings.

In the following, we denote the units of $\mathcal{A}$ and $\mathcal{B}$ by $I_1$ and $I_2$, respectively. We write $G=\left[
                                \begin{array}{cc}
                                  A & M \\
                                  0 & B \\
                                \end{array}
                              \right]$
and
\begin{align}
 \phi\left[
       \begin{array}{cc}
         A & M \\
         0 & B  \\
       \end{array}
     \right]
=\left[
   \begin{array}{cc}
f_{11}(A)+g_{11}(M)+h_{11}(B) & f_{12}(A)+g_{12}(M)+h_{12}(B) \\
0 & f_{22}(A)+g_{22}(M)+h_{22}(B)  \\
   \end{array}
 \right].\label{301}
\end{align}

We divide our proof into several steps.

\textbf{Claim 1}
$f_{12}=f_{22}=0$.

Let $S=\left[
          \begin{array}{cc}
            X & M \\
            0 & B \\
          \end{array}
        \right]
$ and $T=\left[
       \begin{array}{cc}
         X^{-1}A & 0 \\
         0 & I_2 \\
       \end{array}
     \right]
$, where $X$ is an invertible element in $\mathcal{A}$.
Since $ST=G$, we have
\begin{align}
\phi(G)&=\phi(S)T \notag\\
&=\left[
\begin{array}{cc}
f_{11}(X)+g_{11}(M)+h_{11}(B) & f_{12}(X)+g_{12}(M)+h_{12}(B) \\
 0 & f_{22}(X)+g_{22}(M)+h_{22}(B) \\
 \end{array}
 \right]\left[
          \begin{array}{cc}
           X^{-1}A & 0 \\
         0 & I_2 \\
          \end{array}
        \right] \notag\\
&=\left[
\begin{array}{cc}
 * & f_{12}(X)+g_{12}(M)+h_{12}(B) \\
 0 & f_{22}(X)+g_{22}(M)+h_{22}(B) \\
 \end{array}
 \right].\label{302}
\end{align}
By comparing \eqref{301} with \eqref{302}, we obtain $f_{12}(X)=f_{12}(A)$ and $f_{22}(X)=f_{22}(A)$ for each invertible element $X$ in $\mathcal{A}$. Noting that $A$ is a fixed element, for any nonzero element $\lambda$ in $\mathbb{C}$, we have $f_{12}(\lambda X)=f_{12}(A)=\lambda f_{12}(X)=\lambda f_{12}(A)$. It follows that $f_{12}(X)=0$ for each invertible element $X$. Thus $f_{12}(X)=0$ for all $X$ in $\mathcal{A}$. Similarly, we can obtain $f_{22}(X)=0$.

\textbf{Claim 2}
$h_{12}=h_{11}=0$.

Let $S=\left[
          \begin{array}{cc}
            I_1 & 0 \\
            0 & BZ^{-1} \\
          \end{array}
        \right]
$ and $T=\left[
       \begin{array}{cc}
         A & M \\
         0 & Z \\
       \end{array}
     \right]
$, where $Z$ is an invertible element in $\mathcal{B}$.
Since $ST=G$, we have
\begin{align}
\phi(G)&=S\phi(T) \notag\\
&=\left[
          \begin{array}{cc}
            I_1 & 0 \\
            0 & BZ^{-1} \\
          \end{array}
        \right]
\left[
\begin{array}{cc}
f_{11}(A)+g_{11}(M)+h_{11}(Z) & f_{12}(A)+g_{12}(M)+h_{12}(Z) \\
 0 & f_{22}(A)+g_{22}(M)+h_{22}(Z) \\
 \end{array}
 \right]
 \notag\\
&=\left[
\begin{array}{cc}
 f_{11}(A)+g_{11}(M)+h_{11}(Z) & f_{12}(A)+g_{12}(M)+h_{12}(Z) \\
 0 & * \\
 \end{array}
 \right].\label{303}
\end{align}
By comparing \eqref{301} with \eqref{303}, we obtain $h_{12}(Z)=h_{12}(B)$ and $h_{11}(Z)=h_{11}(B)$ for each invertible element $Z$ in $\mathcal{B}$. Similarly as the previous discussion, we can obtain $h_{12}(Z)=h_{11}(Z)=0$ for all $Z$ in $\mathcal{B}$.

\textbf{Claim 3}
$g_{22}=g_{11}=0$.

For every $Y$ in $\mathcal{M}$, we set
 $S=\left[
          \begin{array}{cc}
            I_1 & M-Y \\
            0 & B \\
          \end{array}
        \right]
$, $T=\left[
       \begin{array}{cc}
         A & Y \\
         0 & I_2 \\
       \end{array}
     \right]
$. Obviously, $ST=G$. Thus we have
\begin{align}
\phi(G)&=\phi(S)T \notag\\
&=\left[
\begin{array}{cc}
* & * \\
0 & f_{22}(I_1)+g_{22}(M-Y)+h_{22}(B) \\
 \end{array}
 \right]
 \left[
          \begin{array}{cc}
           A & Y \\
         0 & I_2 \\
          \end{array}
        \right] \notag\\
&=\left[
\begin{array}{cc}
* & * \\
0 & f_{22}(I_1)+g_{22}(M-Y)+h_{22}(B) \\
 \end{array}
 \right].\label{304}
\end{align}
By comparing \eqref{301} with \eqref{304}, we obtain $f_{22}(I_1)+g_{22}(M-Y)+h_{22}(B)=f_{22}(A)+g_{22}(M)+h_{22}(B)$. Hence $g_{22}(Y)=f_{22}(I_1-A)$. It means $g_{22}(Y)=0$ immediately.

On the other hand,
\begin{align}
\phi(G)&=S\phi(T) \notag\\
&=\left[
          \begin{array}{cc}
            I_1 & M-Y \\
            0 & B \\
          \end{array}
        \right]
\left[
\begin{array}{cc}
f_{11}(A)+g_{11}(Y)+h_{11}(I_2) & * \\
 0 & * \\
 \end{array}
 \right]
 \notag\\
&=\left[
\begin{array}{cc}
 f_{11}(A)+g_{11}(Y)+h_{11}(I_2) & * \\
 0 & * \\
 \end{array}
 \right].\label{305}
\end{align}
By comparing \eqref{301} with \eqref{305}, we obtain $g_{11}(Y)=g_{11}(M)+h_{11}(B-I_2)$. Hence $g_{11}(Y)=0$.

According to the above three claims, we obtain that
\[
 \phi\begin{bmatrix}
X & Y \\
0 & Z  \\
\end{bmatrix}
=
\begin{bmatrix}
f_{11}(X) & g_{12}(Y) \\
0 & h_{22}(Z)  \\
\end{bmatrix}
\]
for every
$\left[
  \begin{array}{cc}
    X & Y \\
    0 & Z \\
  \end{array}
\right]$
in $\mathcal{J}$.

\textbf{Claim 4}
$f_{11}(X)=f_{11}(I_1)X$ for all $X$ in $\mathcal{A}$, and $g_{12}(Y)=f_{11}(I_1)Y$ for all $Y$ in $\mathcal{M}$.

Let $S=\left[
          \begin{array}{cc}
            X & M-XY \\
            0 & B \\
          \end{array}
        \right]
$ and $T=\left[
       \begin{array}{cc}
         X^{-1}A & Y \\
         0 & I_2 \\
       \end{array}
     \right]
$, where $X$ is an invertible element in $\mathcal{A}$, and $Y$ is an arbitrary element in $\mathcal{M}$.
Since $ST=G$, we have
\begin{align}
\phi(G)&=\phi(S)T \notag\\
&=\left[
          \begin{array}{cc}
            f_{11}(X) & g_{12}(M-XY) \\
            0 & h_{22}(B) \\
          \end{array}
        \right]
\left[
\begin{array}{cc}
         X^{-1}A & Y \\
         0 & I_2 \\
 \end{array}
 \right]
 \notag\\
&=\left[
\begin{array}{cc}
* & f_{11}(X)Y+g_{12}(M-XY) \\
 0 & * \\
 \end{array}
 \right] \notag\\
&=\left[
    \begin{array}{cc}
f_{11}(A) & g_{12}(M) \\
0 & h_{22}(B)  \\
    \end{array}
  \right].\label{306}
\end{align}
So we have $f_{11}(X)Y=g_{12}(XY)$. It follows that
\begin{align}
g_{12}(Y)=f_{11}(I_1)Y\label{307}
\end{align}
by taking $X=I_1$. Replacing $Y$ in \eqref{307} with $XY$, we can obtain $g_{12}(XY)=f_{11}(I_1)XY=f_{11}(X)Y$ for each invertible element $X$ in $\mathcal A$ and $Y$ in $\mathcal M$. Since $\mathcal{M}$ is faithful, we have
\begin{align}
f_{11}(X)=f_{11}(I_1)X\label{308}
\end{align}
for all invertible elements $X$ and so for all elements in $\mathcal{A}$.

\textbf{Claim 5}
$h_{22}(Z)=Zh_{22}(I_2)$ for all $Z$ in $\mathcal{B}$, and $g_{12}(Y)=Yh_{22}(I_2)$ for all $Y$ in $\mathcal{M}$.

Let $S=\left[
          \begin{array}{cc}
            I_1 & Y \\
            0 & BZ^{-1} \\
          \end{array}
        \right]
$ and $T=\left[
       \begin{array}{cc}
         A & M-YZ \\
         0 & Z \\
       \end{array}
     \right]
$, where $Z$ is an invertible element in $\mathcal{B}$, and $Y$ is an arbitrary element in $\mathcal{M}$.
Since $ST=G$, we have
\begin{align}
\phi(G)&=S\phi(T) \notag\\
&=\left[
          \begin{array}{cc}
            I_1 & Y \\
            0 & BZ^{-1} \\
          \end{array}
        \right]
\left[
\begin{array}{cc}
            f_{11}(A) & g_{12}(M-YZ) \\
            0 & h_{22}(Z) \\
 \end{array}
 \right]
 \notag\\
&=\left[
\begin{array}{cc}
* & g_{12}(M-YZ)+Yh_{22}(Z) \\
 0 & * \\
 \end{array}
 \right] \notag\\
&=\left[
    \begin{array}{cc}
f_{11}(A) & g_{12}(M) \\
0 & h_{22}(B)  \\
    \end{array}
  \right].\label{309}
\end{align}
So we have $g_{12}(YZ)=Yh_{22}(Z)$. Through a similar discussion as the proof of Claim 4, we obtain
$h_{22}(Z)=Zh_{22}(I_2)$
for all $Z$ in $\mathcal{B}$
and
$g_{12}(Y)=Yh_{22}(I_2)$
for all $Y$ in $\mathcal{M}$.

Thus we have that
\[
 \phi\begin{bmatrix}
X & Y \\
0 & Z  \\
\end{bmatrix}
=
\begin{bmatrix}
f_{11}(I_1)X & f_{11}(I_1)Y \\
0 & Zh_{22}(I_2)  \\
\end{bmatrix}
=
\begin{bmatrix}
f_{11}(I_1)X & Yh_{22}(I_2) \\
0 & Zh_{22}(I_2)  \\
\end{bmatrix}
\]
for every
$\left[
  \begin{array}{cc}
    X & Y \\
    0 & Z \\
  \end{array}
\right]$
in $\mathcal{J}$.
So it is sufficient to show that $f_{11}(I_1)X=Xf_{11}(I_1)$ for all $X$ in $\mathcal{A}$, and $h_{22}(I_2)Z=Zh_{22}(I_2)$ for all $Z$ in $\mathcal{B}$.
Since $f_{11}(I_1)Y=Yh_{22}(I_2)$ for all $Y$ in $\mathcal{M}$, we have $f_{11}(I_1)XY=XYh_{22}(I_2)=Xf_{11}(I_1)Y$. It implies that $f_{11}(I_1)X=Xf_{11}(I_1)$. Similarly, $h_{22}(I_2)Z=Zh_{22}(I_2)$. Now we can obtain that $\phi(J)=\phi(I)J=J\phi(I)$ for all $J$ in $\mathcal{J}$, where
$I=\left[
      \begin{array}{cc}
        I_1 & 0 \\
        0 & I_2 \\
      \end{array}
    \right]$
is the unit of $\mathcal{J}$.
Hence, $G$ is a full-centralizable point.
\end{proof}

As applications of Theorem 3.1, we have the following corollaries.
    \begin{corollary}
    Let $\mathcal{A}$ be a nest algebra on a Hilbert space $\mathcal H$. Then every element in $\mathcal{A}$ is a full-centralizable point.
    \end{corollary}
\begin{proof}
If $\mathcal{A}=B(\mathcal H)$, then the result follows from Theorem 2.1. Otherwise, $\mathcal A$ is isomorphic to a triangular algebra. By Theorem 3.1, the result follows.
\end{proof}

    \begin{corollary}
    Let $\mathcal{A}$ be a CDCSL(completely distributive commutative subspace lattice) algebra on a Hilbert space $\mathcal H$. Then every element in $\mathcal{A}$ is a full-centralizable point.
    \end{corollary}

\begin{proof}
It is known that $\mathcal{A}\cong \sum\limits_{i \in \Lambda}\bigoplus\mathcal{A}_i$, where each $\mathcal{A}_i$ is either $B(\mathcal H_i)$ for some Hilbert space $\mathcal H_i$ or a triangular algebra
$Tri(\mathcal{B},\mathcal{M},\mathcal{C})$ such that the conditions of Theorem 3.1 hold(see in \cite{517} and \cite{518}). By Lemma 2.2, the result follows.
\end{proof}

\textbf{Remark} For the definition of a CDCSL algebra, we refer to \cite{521}.
\section{Derivations on von Neumann algebras}\

In this section, we characterize the derivable mappings at a given point in a von Neumann algebra.

\begin{lemma}
Let $\mathcal{A}$ be a von Neumann algebra. Suppose $\Delta:\mathcal{A}\rightarrow\mathcal{A}$ is a linear mapping such that $\Delta(A)B+A\Delta(B)=0$ for each $A$ and $B$ in $\mathcal{A}$ with $AB=0$. Then $\Delta=D+\phi$, where $D:\mathcal{A}\rightarrow\mathcal{A}$ is a derivation, and $\phi:\mathcal{A}\rightarrow\mathcal{A}$ is a centralizer. In particular, $\Delta$ is bounded.
\end{lemma}

\begin{proof}
\textbf{Case 1} $\mathcal{A}$ is an abelian von Neumann algebra. In this case, $\mathcal{A}\cong C(\mathcal{X})$ for some compact Hausdorff space $\mathcal{X}$. If $AB=0$, then the supports of $A$ and $B$ are disjoint. So the equation $\Delta(A)B+A\Delta(B)=0$ implies that $\Delta(A)B=A\Delta(B)=0$. By Lemma 2.5, $\Delta$ is a centralizer.

\textbf{Case 2} $\mathcal{A}\cong M_n(\mathcal{B})( n\geq2)$, where $\mathcal{B}$ is also a von Neumann algebra. By \cite[Theorem 2.3]{513}, $\Delta$ is a generalized derivation with $\Delta(I)$ in the center. That is to say, $\Delta$ is a sum of a derivation and a centralizer.

For general cases, we know $\mathcal{A}\cong\sum\limits_{i=1}^{n}\bigoplus\mathcal{A}_i$, where each $\mathcal{A}_i$ coincides with either Case 1 or Case 2. We write
$A=\sum\limits_{i=1}^{n}A_i$ with $A_i\in\mathcal{A}_i$ and denote the restriction of $\Delta$ in $\mathcal{A}_i$ by $\Delta_i$. It is not difficult to check that $\Delta(A_i) \in \mathcal{A}_i$.
Moreover, setting $A_iB_i=0$, we have $\Delta(A_i)B_i+A_i\Delta(B_i)=\Delta_i(A_i)B_i+A_i\Delta_i(B_i)=0$. By Case 1 and Case 2, each $\Delta_i$ is a sum of a derivation and a centralizer. Hence, $\Delta=\sum\limits_{i=1}^{n}\Delta_i$ is a sum of a derivation and a centralizer.
\end{proof}

\textbf{Remark} In \cite{520}, the authors prove that for a prime semisimple Banach algebra $\mathcal A$ with nontrival idempotents and a linear mapping $\Delta$ from $\mathcal A$ into itself, the condition $\Delta(A)B+A\Delta(B)=0$ for each $A$ and $B$ in $\mathcal{A}$ with $AB=0$ implies that $\Delta$ is bounded. By Lemma 4.1, we have that for a von Neumann algebra $\mathcal A$, the result holds still even if $\mathcal A$ is not prime.

Now we  prove our main result in this section.
\begin{theorem}
Let $\mathcal{A}$ be a von Neumann algebra acting on a Hilbert space $\mathcal{H}$, and $G$ be a given point in $\mathcal{A}$. If $\Delta:\mathcal{A}\rightarrow\mathcal{A}$ is a linear mapping derivable at $G$, then $\Delta=D+\phi$, where $D$ is a derivation, and $\phi$ is a centralizer. Moreover, $G$ is a full-derivable point if and only if $\mathcal C (G)=I$.
\end{theorem}

\begin{proof}
Suppose the range projection of $G$ is $P$. We note that $\mathcal C (G)=\mathcal C (P)$.

Set $Q_1=I-\mathcal{C}(I-P)$, $Q_2=I-\mathcal{C}(P)$, and $Q_3=I-Q_1-Q_2$. Then we have $\mathcal{A}=\sum\limits_{i=1}^{3}\bigoplus\mathcal{A}_i=\sum\limits_{i=1}^{3}\bigoplus(Q_i\mathcal{A})$. For every $A$ in $\mathcal{A}$, we write
$A=\sum\limits_{i=1}^{3}A_i=\sum\limits_{i=1}^{3}Q_iA$.

For any central projection $Q$, setting $Q^{\bot}=I-Q$, we have $$(Q^{\bot}+t^{-1}QGA^{-1})(Q^\bot G+tQA)=G,$$ where $A$ is an arbitrary invertible element in $\mathcal{A}$, and $t$ is an arbitrary nonzero element in $\mathbb{C}$. So we obtain
$$\Delta(G)=(Q^{\bot}+t^{-1}QGA^{-1})\Delta(Q^{\bot}G+tQA)+\Delta(Q^{\bot}+t^{-1}QGA^{-1})(Q^{\bot}G+tQA).$$
Considering the coefficient of $t$, we obtain $Q^{\bot}\Delta(QA)+\Delta(Q^{\bot})(QA)=0$. Since the ranges of $Q$ and $Q^{\bot}$ are disjoint, it follows that $Q^{\bot}\Delta(QA)=0$ and so $\Delta(QA)\in Q\mathcal{A}.$
Since $Q_i$ are central projections, we have $\Delta(\mathcal{A}_i)\subseteq \mathcal{A}_i.$

Denote the restriction of $\Delta$ to $\mathcal{A}_i$ by $\Delta_i$. Setting $A_iB_i=G_i$, it is not difficult to check that $\Delta_i(G_i)=\Delta(A_i)B_i+A_i\Delta(B_i)$.

Since $Q_1\leq P$, we have $\overline{ran G_1}=\overline{ran Q_1G}=Q_1H$. So $G_1$ is a right separating point in $\mathcal{A}_1$. By \cite[Corallary 2.5]{512}, $\Delta_1$ is a Jordan derivation  and so is a derivation on $\mathcal{A}_1$.

Since $Q_2\leq I-P$, we have $G_2=Q_2G=0$. By Lemma 4.1, $\Delta_2$ is a sum of a derivation and a centralizer on $\mathcal{A}_2$.

Note that $\overline{ran G_3}=\overline{ran Q_3G}=Q_3P=P_3$. As we proved before, $\mathcal{C}_{\mathcal{A}_3}(P_3)=\mathcal{C}_{\mathcal{A}_3}(Q_3-P_3)=Q_3$. So by \cite[Theorem 3.1]{511}, $\Delta_3$ is a derivation on $\mathcal{A}_3$.

Hence, $\Delta=\sum\limits_{i=1}^{3}\Delta_i$ is a sum of a derivation and a centralizer.

If $\mathcal{C}(G)=I$, then $Q_2=0$, $\mathcal{A}=\mathcal{A}_1\bigoplus\mathcal{A}_3$ and $G=G_1+G_3$ is a full-derivable point.
If $\mathcal{C}(G)\neq I$, then $Q_2\neq0$. Define a linear mapping $\delta:\mathcal{A}\rightarrow\mathcal{A}$ by $\delta(A)=A_2$ for all $A\in \mathcal{A}$. One can check that $\delta$ is not a derivation but derivable at $G$. Thus $G$ is not a full-derivable point.
\end{proof}

As an application, we obtain the following corollary.

\begin{corollary}
Let $\mathcal{A}$ be a von Neumann algebra. Then $\mathcal{A}$ is a factor if and only if every nonzero element $G$ in $\mathcal{A}$ is a full-derivable point.
\end{corollary}

\begin{proof}
If $\mathcal{A}$ is a factor, for each nonzero element $G$ in $\mathcal{A}$, we know that $\mathcal C (G)=I$. By Theorem 4.2, $G$ is a full-derivable point.

If $\mathcal{A}$ is not a factor, then there exists a nontrival central projection $P$. Define a linear mapping $\delta:\mathcal{A}\rightarrow\mathcal{A}$ by $\delta(A)=(I-P)A$ for all $A\in \mathcal{A}$. One can check that $\delta$ is not a derivation but derivable at $P$. Thus $P$ is not a full-derivable point.
\end{proof}

\emph{Acknowledgements}. This paper was partially supported by National Natural Science Foundation of China(Grant No. 11371136).

\end{document}